\documentclass[12pt]{amsart}

\usepackage{amsmath}
\usepackage{amssymb}
\usepackage{amsfonts}
\usepackage{amsthm}
\usepackage{enumerate}
\usepackage{hyperref}
\usepackage{color}
\theoremstyle{plain}
\newtheorem{thm}{Theorem}[section]

\newtheorem{prop}[thm]{Proposition}
\newtheorem{lem}[thm]{Lemma}
\newtheorem{cor}[thm]{Corollary}

\theoremstyle{definition}
\newtheorem{dfn}[thm]{Definition}

\newtheorem{exmp}[thm]{Example}

\newtheorem{rem}[thm]{Remark}

\newtheorem{dfns-rems}[thm]{Definitions and Remarks}
\newtheorem{notas-rems}[thm]{Notations and Remarks}
\newtheorem{exmps-rems}[thm]{Examples and Remarks}

\begin{document}

% ------------------------------------------------------------------------

\title[On the Stanley depth  and size of monomial ideals]{On the Stanley depth and size of monomial ideals}

% ------------------------------------------------------------------------

\author[S. A. Seyed Fakhari]{S. A. Seyed Fakhari}

\address{S. A. Seyed Fakhari, School of Mathematics, Institute for Research
in Fundamental Sciences (IPM), P.O. Box 19395-5746, Tehran, Iran.}

\email{fakhari@ipm.ir}

\urladdr{http://math.ipm.ac.ir/fakhari/}

% ------------------------------------------------------------------------

\begin{abstract}
Let $\mathbb{K}$ be a field and $S=\mathbb{K}[x_1,\dots,x_n]$ be the
polynomial ring in $n$ variables over the field $\mathbb{K}$. For every monomial ideal $I\subset S$,  We provide a recursive formula to determine a lower bound for the Stanley depth of $S/I$. We use this formula to prove the inequality ${\rm sdepth}(S/I)\geq {\rm size}(I)$ for a particular class of monomial ideals.
\end{abstract}

% ------------------------------------------------------------------------


% ------------------------------------------------------------------------


% ------------------------------------------------------------------------

%\thanks{This research was in part supported by a grant from IPM (No. 93130422)}

% ------------------------------------------------------------------------

\maketitle

%%%%%%%%%%%%%%%%%%%%%%%%%%%%%%%%%%%%%%%%%%%%%%%%%%%%%%%%%%%%%%%%%%%%%%%%%%

\section{Introduction} \label{sec1}

Let $\mathbb{K}$ be a field and $S=\mathbb{K}[x_1,\dots,x_n]$ be the
polynomial ring in $n$ variables over the field $\mathbb{K}$. Let $M$ be a nonzero
finitely generated $\mathbb{Z}^n$-graded $S$-module. Let $u\in M$ be a
homogeneous element and $Z\subseteq \{x_1,\dots,x_n\}$. The $\mathbb
{K}$-subspace $u\mathbb{K}[Z]$ generated by all elements $uv$ with $v\in
\mathbb{K}[Z]$ is called a {\it Stanley space} of dimension $|Z|$, if it is
a free $\mathbb{K}[\mathbb{Z}]$-module. Here, as usual, $|Z|$ denotes the
number of elements of $Z$. A decomposition $\mathcal{D}$ of $M$ as a finite
direct sum of Stanley spaces is called a {\it Stanley decomposition} of
$M$. The minimum dimension of a Stanley space in $\mathcal{D}$ is called the
{\it Stanley depth} of $\mathcal{D}$ and is denoted by ${\rm
sdepth}(\mathcal {D})$. The quantity $${\rm sdepth}(M):=\max\big\{{\rm
sdepth}(\mathcal{D})\mid \mathcal{D}\ {\rm is\ a\ Stanley\ decomposition\
of}\ M\big\}$$ is called the {\it Stanley depth} of $M$. Stanley \cite{s}
conjectured that $${\rm depth}(M) \leq {\rm sdepth}(M)$$ for all
$\mathbb{Z}^n$-graded $S$-modules $M$. For a reader friendly introduction
to Stanley decomposition, we refer to \cite{psty} and for a nice survey on this topic we refer to \cite{h}.

Let $I$ be a monomial ideal of $S$. In \cite{l}, Lyubeznik associated a numerical invariant to $I$ which is called size and is defined as follows.

\begin{dfn}
Assume that $I$ is a monomial ideal of $S$. Let
$I=\bigcap_{j=1}^s Q_j$ be an irredundant primary decomposition of $I$,
where $Q_j$ ($1\leq j\leq s$) is a monomial ideal of $S$. Let $h$ be the
height of $\sum_{j=1}^s Q_j$, and denote by $v$ the minimum number $t$ such
that there exist $1\leq j_1,\ldots,j_t\leq s$ with $$\sqrt{\sum_{i=1}^t
Q_{j_i}}=\sqrt{\sum_{j=1}^s Q_j}.$$ Then the {\it size} of $I$ is defined
to be $v+n-h-1$.
\end{dfn}

Lyubeznik \cite{l} proved that for every monomial ideal $I$, the inequality ${\rm depth}(I)\geq {\rm size}(I)+1$ holds true. Assuming Stanley's conjecture would be true, one obtains the inequalities ${\rm sdepth}(I)\geq {\rm size}(I)+1$ and ${\rm sdepth}(S/I)\geq {\rm size}(I)$. The first inequality was proved by Herzog, Popescu and Vladoiu for squarefree monomial ideals in \cite{hpv}. Recently, Tang \cite{t} proved the second inequality for squarefree monomial ideals. The aim of this paper is to extend Tang's method to prove the inequality ${\rm sdepth}(S/I)\geq {\rm size}(I)$ for a particular class of monomial ideals containing squarefree monomial ideals.

By \cite[Corollary 1.3.2]{hh'}, a monomial ideal is irreducible if and only if it is generated by pure powers of the variables. Also, by \cite[Theorem 1.3.1]{hh'}, every monomial ideal of $S$ can be written as the intersection of irreducible monomial ideals and every irredundant presentation in this form is unique. Assume that $I=Q_1\cap \ldots \cap Q_s$ is the irredundant presentation of $I$ as the intersection of irreducible monomial ideals. Using this presentation, we provide a recursive formula for computing a lower bound for the Stanley depth of $S/I$ (see Theorem \ref{main}). Assume moreover that for every $1\leq i\leq s$ and every proper nonempty subset $\tau\subset [s]$ with $$\sqrt{Q_i}\subseteq\sum_{j\in \tau}\sqrt{Q_j}$$we have
$$Q_i\subseteq\sum_{j\in\tau}Q_j.$$Then we prove that ${\rm sdepth}(S/I)\geq {\rm size}(I)$ (see Theorem \ref{size}).

Before beginning the proof, we mention that although, the behavior of Stanley depth with polarization is known \cite{ikm}, the following example shows that one can not use the polarization and Tang's result to deduce Theorem \ref{size}.

\begin{exmp}
Let $I=(x_1^2, x_2x_3)$ be a monomial ideal of $S=\mathbb{K}[x_1, x_2, x_3]$. Then $I$ satisfies the the assumptions of Theorem \ref{size} and one can easily check that ${\rm size}(I)=1$. Thus, Theorem \ref{size} implies that ${\rm sdepth}(S/I)\geq 1$. On the other hand, by applying polarization on $I$, we obtain the ideal $I^p=(x_1x_4, x_2x_3)$ as a monomial ideal in the polynomial ring $T=\mathbb{K}[x_1, x_2, x_3, x_4]$. One can check that ${\rm size}(I^p)=1$. Now \cite[Corollary 4.4]{ikm} and \cite[Theorem 3.2]{t} imply that ${\rm sdepth}(S/I)={\rm sdepth}(T/I^p)-1\geq 1-1=0$. Note that this inequality is weaker than one obtained by Theorem \ref{size}.
\end{exmp}

%%%%%%%%%%%%%%%%%%%%%%%%%%%%%%%%%%%%%%%%%%%%%%%%%%%%%%%%%%%%%%%%%%%%%%%%%%

\section{Stanley depth and size} \label{sec2}

In this section, we prove the main results of this paper. Using the irredundant primary decomposition of a monomial ideal $I$, we first provide a decomposition for $S/I$ in Corollary \ref{sidecomp}. Then we use this decomposition to obtain a lower bound for the Stanley depth of $S/I$ (see Theorem \ref{main}). This lower bound and an inductive argument help us to prove the inequality ${\rm sdepth}(S/I)\geq {\rm size}(I)$ for a particular class of monomial ideals (see Theorem \ref{size}).

\begin{rem}
We emphasize that every decomposition in this paper is valid only in the category of $\mathbb{K}$-vector spaces and not in the category of $S$-modules.
\end{rem}

To obtain a decomposition for $S/I$, we first need to have decompositions for $S$ and $I$. The following proposition, provides the required decomposition for $S$. Before beginning the proof, we remind that for every subset $S'$ of $S$, the set of monomials belonging to $S'$ is denoted by ${\rm Mon}(S')$. Also, for every monomial $u\in S$, the {\it support} of $u$, denoted by ${\rm Supp}(u)$ is the set of variables which divide $u$.
\begin{prop} \label{sdecomp}
Let $S'=\mathbb{K}[x_1,\dots,x_r]$, $S''=\mathbb{K}[x_{r+1},\dots,x_n]$, $S=\mathbb{K}[x_1,\dots,x_n]$, and $I$ be a monomial ideal of $S$. Assume that
\[
\begin{array}{rl}
I=Q_1\cap \ldots \cap Q_s, \ s\geq 2
\end{array} \tag{$\dagger$} \label{dag}
\]
is the unique irredundant presentation of $I$ as the intersection of irreducible monomial ideals. Suppose that $Q=\sum_{i=1}^sQ_i$. For every proper subset $\tau\subset [s]$, set $$S_{\tau}=\mathbb{K}\Bigg[x_i \ \bigg| \ 1\leq i \leq r, \ x_i\notin \sum_{j\in \tau}\sqrt{Q_j}\Bigg]$$and$$\mathcal{M}_{\tau}=\Bigg\{u \ \bigg| \ u\in {\rm Mon}(S')\setminus \sum_{j\in \tau}Q_j \Bigg\}\bigcap \mathbb{K}\Bigg[x_i \ \bigg| \ x_i\in \sum_{j\in \tau}\sqrt{Q_j}\Bigg].$$Then$$(\ast) \ \ \ \ \ \ \ \ \ \ \ \ \ S=\Big(\bigoplus_{u\in {\rm Mon}(S'\setminus Q)}uS''\Big)\oplus\Bigg(\bigoplus_{\tau\subset [s]}\bigoplus_{w\in \mathcal{M}_{\tau}}\Bigg(\Big(\bigcap_{j\in [s]\setminus \tau}Q_j\cap wS_{\tau}\Big)S_{\tau}[x_{r+1}, \ldots, x_n]\Bigg)\Bigg).$$
\end{prop}

\begin{proof}
We first prove that every monomial of $S$ belongs to the right hand side of ($\ast$). Let $\alpha \in S$ be a monomial. Then there exist monomials $u\in S'$ and $v\in S''$ such that $\alpha=uv$. If $u\notin Q$, then since $\alpha \in uS''$, it belongs to the first summand. Thus, assume that $u\in Q$.

Let $\tau=\{i\in [s] \mid u\notin Q_i\}$. Since $u\in Q$, it follows that $\tau$ is a proper subset of $[s]$. Now there exist monomials $$w\in \mathbb{K}\Bigg[x_i \ \bigg| \ 1\leq i \leq r, \ x_i\in \sum_{j\in \tau}\sqrt{Q_j}\Bigg] \ \ \ \ {\rm and} \ \ \ \ w'\in S_{\tau}$$such that $u=ww'$. Since for every $j\in\tau$, we have $u\notin Q_j$, it follows that $w\notin Q_j$, for every $j\in \tau$. This shows that $w\in \mathcal{M}_{\tau}$. On the other hand, $u\in \bigcap_{j\in [s]\setminus \tau}Q_j$ and hence $u\in \bigcap_{j\in [s]\setminus \tau}Q_j\cap wS_{\tau}$. Therefore$$\alpha=uv\in \Big(\bigcap_{j\in [s]\setminus \tau}Q_j\cap wS_{\tau}\Big)S_{\tau}[x_{r+1}, \ldots, x_n].$$ It turns out that$$S=\sum_{u\in {\rm Mon}(S'\setminus Q)}uS''+\sum_{\tau\subset [s]}\sum_{w\in \mathcal{M}_{\tau}}\Bigg(\Big(\bigcap_{j\in [s]\setminus \tau}Q_j\cap wS_{\tau}\Big)S_{\tau}[x_{r+1}, \ldots, x_n]\Bigg).$$

We now show that the sum is direct. We consider the following cases.\\

{\bf Case 1.} For every pair of monomials $u_1, u_2 \in S'\setminus Q$, we have $u_1S''\cap u_2S''=0$, since$$S''\cap{\rm Supp}(u_1)=S''\cap{\rm Supp}(u_2)=\emptyset.$$\\

{\bf Case 2.} We prove that for every subset $\tau$ of $[s]$ and every pair of monomials $u \in S'\setminus Q$ and $w\in \mathcal{M}_{\tau}$, we have$$uS''\cap \Bigg(\Big(\bigcap_{j\in [s]\setminus \tau}Q_j\cap wS_{\tau}\Big)S_{\tau}[x_{r+1}, \ldots, x_n]\Bigg)=0.$$

Indeed, assume by the contrary that there exists a monomial $$v\in uS''\cap \Bigg(\Big(\bigcap_{j\in [s]\setminus \tau}Q_j\cap wS_{\tau}\Big)S_{\tau}[x_{r+1}, \ldots, x_n]\Bigg).$$Let $v'$ be the monomial obtained from $v$ by applying the map $x_i\mapsto 1$, for every $r+1\leq i\leq n$. Then $v'=u$ and on the other hand,$$v'\in\bigcap_{j\in [s]\setminus \tau}Q_j\cap wS_{\tau}.$$Therefore, $u\in \bigcap_{j\in [s]\setminus \tau}Q_j$, which is a contradiction by $u\notin Q$.\\

{\bf Case 3.} We prove that for every subset $\tau$ of $[s]$ and every pair of distinct monomials $w_1, w_2\in \mathcal{M}_{\tau}$,$$\Bigg(\Big(\bigcap_{j\in [s]\setminus \tau}Q_j\cap w_1S_{\tau}\Big)S_{\tau}[x_{r+1}, \ldots, x_n]\Bigg)\cap \Bigg(\Big(\bigcap_{j\in [s]\setminus \tau}Q_j\cap w_2S_{\tau}\Big)S_{\tau}[x_{r+1}, \ldots, x_n]\Bigg)=0.$$

Indeed, assume by the contrary that there exists a monomial $$v\in \Bigg(\Big(\bigcap_{j\in [s]\setminus \tau}Q_j\cap w_1S_{\tau}\Big)S_{\tau}[x_{r+1}, \ldots, x_n]\Bigg)\cap \Bigg(\Big(\bigcap_{j\in [s]\setminus \tau}Q_j\cap w_2S_{\tau}\Big)S_{\tau}[x_{r+1}, \ldots, x_n]\Bigg).$$Let $v'$ be the monomial obtained from $v$ by applying the map $x_i\mapsto 1$, for every $i$ with $x_i\in S_{\tau}[x_{r+1}, \ldots,x_n]$. Since $v\in w_1S_{\tau}[x_{r+1}, \ldots, x_n]$ and$$w_1\in \mathbb{K}\Bigg[x_i \ \bigg| \ 1\leq i \leq r, \ x_i\in \sum_{j\in \tau}\sqrt{Q_j}\Bigg],$$we conclude that $v'=w_1$. Similarly $v'=w_2$, which implies that $w_1=w_2$ and this is a contradiction.\\

{\bf Case 4.} We prove that for every pair of proper subsets $\tau_1, \tau_2$ of $[s]$ with $\tau_1\neq \tau_2$ and every pair of monomials $w_1 \in \mathcal{M}_{\tau_1}$ and $w_2 \in \mathcal{M}_{\tau_2}$,$$\Bigg(\Big(\bigcap_{j\in [s]\setminus \tau_1}Q_j\cap w_1S_{\tau_1}\Big)S_{\tau_1}[x_{r+1}, \ldots, x_n]\Bigg)\cap \Bigg(\Big(\bigcap_{j\in [s]\setminus \tau_2}Q_j\cap w_2S_{\tau_2}\Big)S_{\tau_2}[x_{r+1}, \ldots, x_n]\Bigg)=0.$$

Indeed, assume by the contrary that there exists a monomial $$v\in \Bigg(\Big(\bigcap_{j\in [s]\setminus \tau_1}Q_j\cap w_1S_{\tau_1}\Big)S_{\tau_1}[x_{r+1}, \ldots, x_n]\Bigg)\cap \Bigg(\Big(\bigcap_{j\in [s]\setminus \tau_2}Q_j\cap w_2S_{\tau_2}\Big)S_{\tau_2}[x_{r+1}, \ldots, x_n]\Bigg).$$Since $\tau_1\neq \tau_2$, without lose of generality we may assume that $\tau_1\nsubseteq \tau_2$. Thus, there exists an integer $j_0\in \tau_1\setminus \tau_2$. Let $v'$ be the monomial obtained from $v$ by applying the map $x_i\mapsto 1$, for every $r+1\leq i\leq n$. Then$$v'\in \Big(\bigcap_{j\in [s]\setminus \tau_1}Q_j\cap w_1S_{\tau_1}\Big)\cap \Big(\bigcap_{j\in [s]\setminus \tau_2}Q_j\cap w_2S_{\tau_2}\Big),$$in particular $v'\in Q_{j_0}$. On the other hand, by $v'\in w_1S_{\tau_1}$, we conclude that there exists a monomial $w_0\in S_{\tau_1}$, such that $v'=w_0w_1$. Since $w_1\in \mathcal{M}_{\tau_1}$, we see that $w_1\notin Q_{j_0}$. Also, by the definition of  $S_{\tau_1}$, we conclude that $w_0\notin \sqrt{Q_{j_0}}$. Since $Q_{j_0}$ is a primary ideal, $v'=w_0w_1\notin Q_{j_0}$, which is a contradiction. This completes the proof of the proposition.
\end{proof}

\begin{rem} \label{empty}
Notice that in the decomposition of Proposition \ref{sdecomp}, the summand corresponding to $\tau=\emptyset$ is equal to $(I\cap S')S$. Because $\mathcal{M}_{\emptyset}=\{1\}$ and $S_{\emptyset}=S'$.
\end{rem}

In the following proposition, we provide a decomposition for $I$.

\begin{prop} \label{idecomp}
Under the assumptions as in Proposition \ref{sdecomp}, suppose further that one of the irreducible monomial ideals in the decomposition \ref{dag} of $I$ is $(x_1^{a_1}, \ldots, x_r^{a_r})$, where $a_1, \ldots, a_r$ are positive integers. Then there is a decomposition of $I$:

\begin{align*}
I & =\Big((I\cap S')S\Big)\oplus\\ & \bigoplus_{\tau\subset [s]}\bigoplus_{w\in \mathcal{M}_{\tau}}\Bigg(\Bigg(\Big(\bigcap_{j\in [s]\setminus \tau}Q_j\cap wS_{\tau}\Big)S_{\tau}[x_{r+1}, \ldots, x_n]\Bigg) \cap \Bigg(\Big(\bigcap_{j\in \tau}Q_j\cap wS''\Big)S_{\tau}[x_{r+1}, \ldots, x_n]\Bigg)\Bigg),
\end{align*}
where $\tau$ runs over all nonempty proper subsets of $[s]$.
\end{prop}
\begin{proof}
It is clear that every monomial of the sum
\begin{align*}
& \Big((I\cap S')S\Big)+\\ & \sum_{\tau\subset [s]}\sum_{w\in \mathcal{M}_{\tau}}\Bigg(\Bigg(\Big(\bigcap_{j\in [s]\setminus \tau}Q_j\cap wS_{\tau}\Big)S_{\tau}[x_{r+1}, \ldots, x_n]\Bigg) \cap \Bigg(\Big(\bigcap_{j\in \tau}Q_j\cap wS''\Big)S_{\tau}[x_{r+1}, \ldots, x_n]\Bigg)\Bigg),
\end{align*}
belongs to $I$. Thus, we prove that every monomial of $I$ belongs to the above sum. Assume that $\alpha\in I$ is a monomial. Then there exist monomials $u_1\in S'$ and $u_2\in S''$ such that $\alpha=u_1u_2$. Since $I\subseteq (x_1^{a_1}, \ldots, x_r^{a_r})$, we conclude that $u_1\in (x_1^{a_1}, \ldots, x_r^{a_r})\subseteq Q$ and hence$$\alpha\notin\bigoplus_{u\in {\rm Mon}(S'\setminus Q)}uS''.$$Therefore, Proposition \ref{sdecomp} shows that there exists a proper subset $\tau$ of $[s]$ and a monomial $w\in \mathcal{M}_{\tau}$ such that$$\alpha \in \Big(\bigcap_{j\in [s]\setminus \tau}Q_j\cap wS_{\tau}\Big)S_{\tau}[x_{r+1}, \ldots, x_n].$$If $\tau=\emptyset$, then Remark \ref{empty} implies that $\alpha \in (I\cap S')S$.

Thus, assume that $\tau \neq \emptyset$. It is sufficient to prove that $$\alpha \in (\Big(\bigcap_{j\in \tau}Q_j\cap wS''\Big)S_{\tau}[x_{r+1}, \ldots, x_n]\Bigg).$$Remind that $\alpha=u_1u_2$, where $u_1\in S'$ and $u_2\in S''$. It is clear that $u_1 \in wS_{\tau}$. Therefore, there exists a monomial $u'\in S_{\tau}$ such that $u_1=wu'$. Hence $\alpha =wu'u_2$. It follows from the definition of $S_{\tau}$ that for every $j\in \tau$, we have $u'\notin \sqrt{Q_j}$. Since for every $j\in \tau$, we have $\alpha \in I\subseteq Q_j$ and $Q_j$ is a primary ideal, we conclude that $wu_2\in \bigcap_{j\in \tau}Q_j$. This shows that $wu_2\in \bigcap_{j\in \tau}Q_j \cap wS''$. Hence$$\alpha =wu'u_2\in (\Big(\bigcap_{j\in \tau}Q_j\cap wS''\Big)S_{\tau}[x_{r+1}, \ldots, x_n]\Bigg),$$ and it implies that
\begin{align*}
& \Big((I\cap S')S\Big)+\\ & \sum_{\tau\subset [s]}\sum_{w\in \mathcal{M}_{\tau}}\Bigg(\Bigg(\Big(\bigcap_{j\in [s]\setminus \tau}Q_j\cap wS_{\tau}\Big)S_{\tau}[x_{r+1}, \ldots, x_n]\Bigg) \cap \Bigg(\Big(\bigcap_{j\in \tau}Q_j\cap wS''\Big)S_{\tau}[x_{r+1}, \ldots, x_n]\Bigg)\Bigg),
\end{align*}

It now follows from Proposition \ref{sdecomp} that the sum is in fact direct sum.
\end{proof}

The following corollary is an immediate consequence of Propositions \ref{sdecomp}, \ref{idecomp} and Remark \ref{empty}. It provides a decomposition for $S/I$ and helps us to determine a lower bound for the Stanley depth of $S/I$.

\begin{cor} \label{sidecomp}
Under the assumptions as in Proposition \ref{sdecomp}, suppose further that one of the irreducible monomial ideals in the decomposition \ref{dag} of $I$ is $(x_1^{a_1}, \ldots, x_r^{a_r})$, where $a_1, \ldots, a_r$ are positive integers. Then there is a decomposition of $S/I$:
\begin{align*}
& S/I=\Big(\bigoplus_{u\in {\rm Mon}(S'\setminus Q)}uS''\Big)\oplus\\ & \bigoplus_{\tau\subset [s]}\bigoplus_{w\in \mathcal{M}_{\tau}}\frac{\Bigg(\Big(\bigcap_{j\in [s]\setminus \tau}Q_j\cap wS_{\tau}\Big)S_{\tau}[x_{r+1}, \ldots, x_n]\Bigg)}{\Bigg(\Big(\bigcap_{j\in [s]\setminus \tau}Q_j\cap wS_{\tau}\Big)S_{\tau}[x_{r+1}, \ldots, x_n]\Bigg)\cap \Bigg(\Big(\bigcap_{j\in \tau}Q_j\cap wS''\Big)S_{\tau}[x_{r+1}, \ldots, x_n]\Bigg)},
\end{align*}
where $\tau$ runs over all nonempty proper subsets of $[s]$.
\end{cor}

The following lemma is a modification of \cite[Lemma 2.3]{t}. In fact, for $w=1$, it implies \cite[Lemma 2.3]{t}. Using this lemma, we are able to find a lower bound for the Stanley depth of summands appearing in Corollary \ref{sidecomp}.

\begin{lem} \label{ineq}
Let $S_1=\mathbb{K}[x_1, \ldots, x_n]$ and $S_2=\mathbb{K}[y_1, \ldots, y_m]$ be polynomial rings with disjoint set of variables and assume that $S_3=\mathbb{K}[x_1, \ldots, x_n, y_1, \ldots, y_m]$. Assume also that $S=\mathbb{K}[x_1, \ldots, x_n, y_1, \ldots, y_m, z_1, \ldots, z_t]$ is a polynomial ring containing $S_3$. Suppose that $I, J\subset S$ are monomial ideals and $w\in S\setminus J$ is a monomial. Set $I_1=I\cap wS_1$ and $J_1=J\cap wS_2$. Then $${\rm sdepth}_{S_3}\Big(\frac{I_1S_3}{I_1S_3\cap J_1S_3}\Big)\geq {\rm sdepth}_{S_1}\Big((I:w)\cap S_1\Big)+{\rm sdepth}_{S_2}\Big(\frac{S_2}{(J:w)\cap S_2}\Big).$$
\end{lem}
\begin{proof}
We note that every monomial in $I_1S_3$ is divisible by $w$. Thus, the $S_3$-modules $I_1S_3/(I_1S_3\cap J_1S_3)$ and $(I_1S_3:w)/((I_1S_3:w)\cap (J_1S_3:w))$ are isomorphic. Hence, $${\rm sdepth}_{S_3}\Big(\frac{I_1S_3}{I_1S_3\cap J_1S_3}\Big)={\rm sdepth}_{S_3}\Big(\frac{(I_1S_3:w)}{(I_1S_3:w)\cap (J_1S_3:w)}\Big).$$Moreover, by the definition of $I_1$ and $J_1$ we have $(I_1S_3:w)=((I_1S_3:w)\cap S_1)S_3$ and $(J_1S_3:w)=((J_1S_3:w)\cap S_2)S_3$. Therefore, it follows from \cite[Lemma 2.3]{t} and the above inequality that $${\rm sdepth}_{S_3}\Big(\frac{I_1S_3}{I_1S_3\cap J_1S_3}\Big)\geq {\rm sdepth}_{S_1}\Big((I_1S_3:w)\cap S_1\Big)+{\rm sdepth}_{S_2}\Big(\frac{S_2}{(J_1S_3:w)\cap S_2}\Big).$$ Since $(I_1S_3:w)\cap S_1=(I:w)\cap S_1$ and $(J_1S_3:w)\cap S_2=(J:w)\cap S_2$, the assertion follows.
\end{proof}

In the following theorem, we determine a lower bound for the Stanley depth of $S/I$. It is a generalization of \cite[Theorem 2.4]{t}.

\begin{thm} \label{main}
Under the assumptions as in Corollary \ref{sidecomp}, there is an inequality
\begin{align*}
{\rm sdepth}_S(S/I)\geq {\rm min}\Bigg\{ & n-r, {\rm sdepth}_{S_{\tau}}\Big(\bigcap_{j\in [s]\setminus \tau}(Q_j:w)\cap S_{\tau}\Big)\\ &  + {\rm sdepth}_{S''}\Bigg(S''\Big/\Big(\bigcap_{j\in \tau}Q_j\cap S''\Big)\Bigg)\Bigg\},
\end{align*}
where the minimum is taking over all nonempty proper subset $\tau \subset [s]$ and all $w\in \mathcal{M}_{\tau}$ such that $(\cap_{j\in [s]\setminus \tau}Q_j\cap wS_{\tau})\neq 0$.
\end{thm}
\begin{proof}
Note that for every nonempty proper subset $\tau \subset [s]$ and every $w\in \mathcal{M}_{\tau}$, we have $w\notin Q_j$, for all $j\in \tau$. Also, ${\rm Supp}(w)\cap S''=\emptyset$. This shows that for every $j\in \tau$, we have $(Q_j:w)\cap S''=Q_j\cap S''$. Now, the assertion follows from Corollary \ref{sidecomp} and Lemma  \ref{ineq}. To apply Lemma \ref{ineq}, for every summand appearing in Corollary \ref{sidecomp}, set $I=\cap_{j\in [s]\setminus \tau}Q_j$, $J=\cap_{j\in \tau}Q_j$, $S_1=S_{\tau}$, $S_2=S''$ and $S_3=S_{\tau}[x_{r+1}, \ldots, x_n]\subseteq S$.
\end{proof}

We are now ready to prove the main result of this paper.

\begin{thm} \label{size}
Let $I$ be a monomial ideal of $S$. Assume that
\[
\begin{array}{rl}
I=Q_1\cap \ldots \cap Q_s, \ s\geq 2
\end{array}
\]
is the unique irredundant presentation of $I$ as the intersection of irreducible monomial ideals. Suppose that for every $1\leq i\leq s$ and every proper nonempty subset $\tau\subset [s]$ with$$\sqrt{Q_i}\subseteq\sum_{j\in \tau}\sqrt{Q_j}$$we have$$Q_i\subseteq\sum_{j\in \tau}Q_j.$$
Then ${\rm sdepth}(S/I)\geq {\rm size}_S(I).$
\end{thm}
\begin{proof}
We prove the assertion by induction on $s$. Without loss of generality assume that $Q_1=(x_1^{a_1}, \ldots, x_r^{a_r})$, for some integer $r$ with $1\leq r\leq n$. If $s=1$. Then $I=Q_1$ and it is clear that ${\rm size}_S(I)=n-r$. On the other hand, it follows from \cite[Theorem 1.1]{r} that ${\rm sdepth}(S/I)=n-r$. Thus, there is nothing to prove in this case. Hence assume that $s\geq 2$.

Set $S'=\mathbb{K}[x_1,\dots,x_r]$ and $S''=\mathbb{K}[x_{r+1},\dots,x_n]$. It is obvious from the definition of size that ${\rm size}_S(I)\leq n-r$. Therefore, using Theorem \ref{main}, it is enough to prove that for every nonempty proper subset $\tau \subset [s]$ and every $w\in \mathcal{M}_{\tau}$ with $(\cap_{j\in [s]\setminus \tau}Q_j\cap wS_{\tau})\neq 0$, we have$${\rm sdepth}_{S_{\tau}}\Big(\bigcap_{j\in [s]\setminus \tau}(Q_j:w)\cap S_{\tau}\Big) + {\rm sdepth}_{S''}\Bigg(S''\Big/\Big(\bigcap_{j\in \tau}Q_j\cap S''\Big)\Bigg)\geq {\rm size}_S(I).$$ Hence, we fix a nonempty proper subset $\tau \subset [s]$ and a monomial $w\in \mathcal{M}_{\tau}$ such that $(\cap_{j\in [s]\setminus \tau}Q_j\cap wS_{\tau})\neq 0$. If $\bigcap_{j\in \tau}Q_j\cap S''=0$, then
\begin{align*}
& {\rm sdepth}_{S_{\tau}}\Big(\bigcap_{j\in [s]\setminus \tau}(Q_j:w)\cap S_{\tau}\Big) + {\rm sdepth}_{S''}\Bigg(S''\Big/\Big(\bigcap_{j\in \tau}Q_j\cap S''\Big)\Bigg)\\ & \geq n-r\geq {\rm size}_S(I).
\end{align*}
Thus, assume that $\bigcap_{j\in \tau}Q_j\cap S''\neq 0$. In particular $1\notin \tau$. If $S_{\tau}=\mathbb{K}$, then it follows from the definition of $S_{\tau}$ that$$\sqrt{Q_1}\subseteq \sum_{j\in \tau}\sqrt{Q_j}.$$ Hence, by assumption $$Q_1\subseteq \sum_{j\in \tau}Q_j.$$Since $S_{\tau}=\mathbb{K}$, it follows from $(\cap_{j\in [s]\setminus \tau}Q_j\cap wS_{\tau})\neq 0$ and the above inclusion that$$w\in \cap_{j\in [s]\setminus \tau}Q_j\subseteq Q_1\subseteq \sum_{j\in \tau}Q_j,$$which is a contradiction by the definition of $\mathcal{M}_{\tau}$. Therefore, assume that $S_{\tau}\neq\mathbb{K}$. In other words $S_{\tau}$ is a polynomial ring of positive dimension.

Since $(\cap_{j\in [s]\setminus \tau}Q_j\cap wS_{\tau})\neq 0$, we conclude that $\bigcap_{j\in [s]\setminus \tau}(Q_j:w)\cap S_{\tau}$ is a nonzero ideal of $S_{\tau}$. It follows from \cite[Corollary 24]{h} that$${\rm sdepth}_{S_{\tau}}\Big(\bigcap_{j\in [s]\setminus \tau}(Q_j:w)\Big)\geq 1.$$Also, for every $i\in\tau$ and every proper subset $\tau'\subset \tau$, with$$\sqrt{Q_i\cap S''}\subseteq\sum_{j\in \tau'}\sqrt{Q_j\cap S''}$$we have$$\sqrt{Q_i}\subseteq\sum_{j\in \tau'\cup\{1\}}\sqrt{Q_j}$$and the assumption implies that$$Q_i\subseteq\sum_{i\in \tau'\cup\{1\}}Q_j.$$Thus,$$Q_i\cap S''\subseteq\sum_{i\in \tau'}Q_j\cap S''.$$
Thus, the induction hypothesis together with \cite[Lemma 3.2]{hpv} implies that
\begin{align*}
& {\rm sdepth}_{S_{\tau}}\Big(\bigcap_{j\in [s]\setminus \tau}(Q_j:w)\cap S_{\tau}\Big) + {\rm sdepth}_{S''}\Bigg(S''\Big/\Big(\bigcap_{j\in \tau}Q_j\cap S''\Big)\Bigg)\\ & \geq 1+{\rm size}_{S''}\Big(\bigcap_{j\in \tau}Q_j\cap S''\Big)\geq {\rm size}_S(I).
\end{align*}
\end{proof}

\begin{rem}
\begin{enumerate}
\item Every squarefree monomial ideal satisfies the assumption of Theorem \ref{size}. Because $Q_i=\sqrt{Q_i}$ for every $1\leq i\leq s$, if $I=Q_1\cap \ldots \cap Q_s$ is a squarefree monomial ideal. Thus, Theorem \ref{size} is an extension of Tang's result \cite[Theorem 3.2]{t}
\item Note that every monomial ideal satisfying the assumption of Theorem \ref{size} has no embedded associated prime. Indeed, assume that $\sqrt{Q_i}\subseteq \sqrt{Q_j}$ for $i\neq j$. Then the assumption of Theorem \ref{size} implies that $Q_i\subseteq Q_j$, which is contradiction. Because the intersection $Q_1\cap \ldots \cap Q_s$ is irredundant.
\end{enumerate}
\end{rem}

%%%%%%%%%%%%%%%%%%%%%%%%%%%%%%%%%%%%%%%%%%%%%%%%%%%%%%%%%%%%%%%%%%%%%%%%%%

%%%%%%%%%%%%%%%%%%%%%%%%%%%%%%%%%%%%%%%%%%%%%%%%%%%%%%%%%%%%%%%%%%%%%%%%%%

\end{document}